\documentclass[letterpaper]{amsart}
\usepackage{hyperref}
\usepackage{amsrefs}
\usepackage{amsthm}
\usepackage{amssymb}
\usepackage{amsmath}
\numberwithin{equation}{section}

\newtheorem{thm}{Theorem}

\newtheorem{conj}[thm]{Conjecture}
\newtheorem{rmk}[thm]{Remark}

\begin{document}

\title{On the equivariant log-concavity for the cohomology of the flag varieties}

\author{Tao Gui}
\address{Academy of Mathematics and Systems Science, Chinese Academy of Sciences, Beijing, China}
\email{guitao18@mails.ucas.ac.cn}

\maketitle

\begin{abstract}
We study the $S_n$-equivariant log-concavity of the cohomology of flag varieties, also known as the coinvariant ring of $S_n$. Using the theory of representation stability, we give computer-assisted proofs of the equivariant log-concavity in low degrees and high degrees and conjecture that it holds for all degrees. Furthermore, we make a stronger unimodal conjecture which implies the equivariant log-concavity.
\end{abstract}

The motivation of this notes is the following conjecture 

\begin{conj} \label{flag}
For all integer $n \geq 1$, the cohomology ring of the flag manifold $U_{n} / T$, which is also known as the coinvariant ring of $S_{n}$: $H^{2 *}(U_{n} / T, \mathbb{C})\cong \mathbb{C}\left[t_{1}, \ldots, t_{n}\right] /\left(\sigma_{1}, \ldots, \sigma_{n}\right)$, is equivariantly log-concave as graded representation of $S_{n}$ in the sense that $H^{2*(m-1)} \otimes H^{2*(m+1)}$ is isomorphic to a subrepresentation of $H^{2*m} \otimes H^{2*m}$ for all $1 \leq m \leq \binom{n}{2}-1$, where the $t_{j}$'s are of degree 2 and the $\sigma_{i}$'s are the elementary symmetric polynomials in the variables $t_{j}$'s.
\end{conj}

\textbf{Acknowledgment}
The author would like to thank Ming Fang, Hongsheng Hu, Nicolai Reshetikhin, Peng Shan, Nanhua Xi, and Rui Xiong for useful discussions, and is grateful to Matthew H.Y. Xie and Arthur L.B. Yang for the help with computer calculations.

\setcounter{tocdepth}{2}

\section{Cohomology of flag variaties: the Borel's picture}
 
We begin with the classical Borel's picture. Consider the full flag variety $\mathcal{F}_{n}$ over $\mathbb{C}$, parametrizing all complete flags in a $n$-dimensional $\mathbb{C}$-vector space $V$: $$
\{0\}=V_{0} \subset V_{1} \subset V_{2} \subset \cdots \subset V_{n}=V.
$$
Written as a homogeneous space, $\mathcal{F}_{n} \cong \mathrm{GL}(n, \mathbb{C}) / B_{n}$, where $B_{n}$ is the group of non-singular upper triangular matrices. In particular, $\mathcal{F}_{n}$ has dimension $n(n-1) / 2$ over $\mathbb{C}$.

The classical Borel's picture \cite{borel1953cohomologie} gives a explicit and canonical description of the cohomology of $\mathcal{F}_{n}$ using the so called coinvariant ring.
\begin{thm}[Borel,  \cite{borel1953cohomologie}]
As a graded ring,
\begin{equation} \label{Borel}
H^{*}(\mathcal{F}_{n}, \mathbb{C}) \cong \mathbb{C}\left[t_{1}, \ldots, t_{n}\right] /\left(\sigma_{1}, \ldots, \sigma_{n}\right),
\end{equation} 
where the $t_{j}$'s are of degree 2 and the $\sigma_{i}$'s are the elementary symmetric polynomials in the variables $t_{j}$'s.
\end{thm}

Use this identification or use the fiber bundle 
\begin{equation}
\begin{aligned} \mathcal{F}_{n-1} \hookrightarrow & \mathcal{F}_{n} \\ & \downarrow \\ & \mathbb{C P}^{n-1} \end{aligned}
\end{equation}  (the vertical map comes from projecting each flag to the containing line) and the Leray–Hirsch theorem, one can easily compute the Betti numbers $b_{2i}$ of $\mathcal{F}_{n}$ (the odd Betti numbers $b_{2i+1}$ are all $0$), which is equal to the cardinality of $\left\{w \in S_{n} \mid l(w)=i\right\}$ by the Bruhat decomposition, where $l(w)$ is the inversion number of the permutation $w$, and the Poincaré polynomial is
\begin{equation}
P_{\mathcal{F}_{n}}(q):=\sum_{i} b_{2i} q^{i}=P_{\mathcal{F}_{n-1}} \cdot\left(1+q+\cdots+q^{n-1}\right)=\prod_{k=0}^{n-1}\left(1+q+\cdots+q^{k}\right).
\end{equation}
It follows that $P_{\mathcal{F}_{n}}(q)$ is a symmetric, unimodal, and log-concave polynomial \footnote{We call a polynomial symmetric (unimodal, log-concave, respectively) if the sequence of coefficients is symmetric (unimodal, log-concave, respectively). It is easy to see that if two polynomials are log-concave with all positive coefficients, so is their product.}.

Since $\mathcal{F}_{n}$ is a smooth complex projective variety, the symmetry of the Betti numbers comes from the Poincaré duality, and the unimodality comes from the hard Lefschetz theorem, but where is the log-concavity comes from?

\section{The action, the graded character, and the graded multiplicity}

Actually, $S_{n}$ as the Weyl group of $\mathrm{GL}(n, \mathbb{C})$ can not act naturally on the flag variety $\mathcal{F}_{n} \cong \mathrm{GL}(n, \mathbb{C}) / B_{n}$. Nevertheless, it is a classical fact that it indeed acts on the cohomology $H^{*}(\mathcal{F}_{n})$. For example, one can use the diffeomorphism
$U_{n} / T \rightarrow \mathrm{GL}(n, \mathbb{C}) / B_{n}$ and the natural action of $S_{n} \cong N_{T} / T$ on $U_{n} / T$ to induce the action of $S_{n}$ on $H^{*}(\mathcal{F}_{n})$, where $U_{n}$ is the unitary group, $T$ is the compact torus, $N_{T}$ is the normalizer of $T$ in $U_{n}$. 

Use this action, the Borel isomorphism $$H^{*}(\mathcal{F}_{n}, \mathbb{C}) \cong \mathbb{C}\left[t_{1}, \ldots, t_{n}\right] /\left(\sigma_{1}, \ldots, \sigma_{n}\right)$$ becomes an isomorphism of graded representations, where $S_{n}$ acts on the right hind side by permuting the $t_{i}$'s. Using equivariant cohomology, it is a classical result that if we forget the grading, $$H^{*}(\mathcal{F}_{n}, \mathbb{C}) \cong \mathbb{C}\left[t_{1}, \ldots, t_{n}\right] /\left(\sigma_{1}, \ldots, \sigma_{n}\right)$$ carries the regular representation of $S_{n}$, see for example, \cite{atiyah2001equivariant}. Therefore, the cohomological degrees cut the regular representation into graded pieces, and what we are interested is the log-concavity of the tensor product corresponding to this grading.

Consider the fiber bundle
\begin{equation}
\begin{aligned} \mathcal{F}_{n} \cong U_{n} / T \hookrightarrow & \mathcal{B}_{T} \\ & \downarrow \\ & \mathcal{B}_{U} \end{aligned},
\end{equation} 
where $\mathcal{B}_{T}$ and $\mathcal{B}_{U}$ are the classifying spaces of $T$ and $U_{n}$, respectively. Taking cohomology, we get 
\begin{equation}
\mathbb{C}\left[t_{1}, \ldots, t_{n}\right] \cong \mathbb{C}\left[t_{1}, \ldots, t_{n}\right] /\left(\sigma_{1}, \ldots, \sigma_{n}\right) \otimes \mathbb{C}\left[\sigma_{1}, \ldots, \sigma_{n}\right]
\end{equation}
as graded representations of $S_{n}$. Therefore, the graded character of $H^{*}(\mathcal{F}_{n})$ is given by
\begin{equation} \label{grcha}
\chi(w, q)=\prod_{i=1}^{n}\left(1-q^{i}\right) / \prod_{i=1}^{k}\left(1-q^{\lambda i}\right),
\end{equation}
where $w \in S_{n}$ is a permutation of cycle type $\lambda=(\lambda_{1}, \cdots, \lambda_{k})$.

It is easy to see that 
\begin{equation} \label{dual}
q^{n(n-1) / 2} \chi(w, q^{-1})=(-1)^{l(w)}\chi(w, q)
\end{equation}
by the graded character (\ref{grcha}), therefore \emph{representations in complement degrees differ by tensoring the sign representation}, which can also be deduced from the Poincare duality. Using the graded character (\ref{grcha}), one can compute the graded multiplicities of $H^{*}(\mathcal{F}_{n})$, and by work of Lusztig and Stanley,
the multiplicities have the following combinatorial interpretation

\begin{thm}(Lusztig--Stanley, \cite[Prop. 4.11]{stanley1979invariants}\footnote{See \cite[Theorem 8.8]{MR1231799} for a proof of this result and the graded character (\ref{grcha}).})
For any partition $\lambda$ of $n$, as long as $i\leq n(n-1) / 2$, the multiplicity of the irreducible representation
$V(\lambda)$ of $S_n$ corresponding to $\lambda$ in $H^{2i}(\mathcal{F}_{n})$ equals the number \footnote{Known as the “fake degree”, see, for example, \cite{billey2020tableau}.} $b_{\lambda, i}$ of standard tableaux of shape $\lambda$ with major index equal to $i$, where the major index of a tableau\footnote{There is also a notion of the major index of a permutation, which is equal to the major index of the recording tableau (the Q-symbol) of that permutation under the Robinson–Schensted–Knuth correspondence.} is the sum of the numbers $j$ so that the box labeled $j+1$ is in a lower row than the box labeled $j$.
\end{thm}

Therefore, using the Frobenius characteristic map, one can transform Conjecture \ref{flag} into a Schur positivity conjecture in the symmetric function theory.

\begin{conj} \label{Schur}
For all natural number $n$ and $1 \leq i \leq \binom{n}{2}-1$, the symmetric function $S_{n, i} * S_{n, i}-S_{n, i-1} * S_{n, i+1}$ is a non-negative linear combination of the Schur polynomials, where $S_{n, i}:=\sum_{\lambda \vdash n} b_{\lambda, i} S_{\lambda}$, where $S_{\lambda}$ is the Schur polynomial corresponding to partition $\lambda$, $*$ denote the internal product of symmetric functions.
\end{conj}

In terms of Kronecker coefficients, which are the structure coefficients of the tensor product of irreducible representations of $S_{n}$, the above conjectures are equivalent to the following numerical inequalities

\begin{conj} \label{kro}
For all natural number $n$ and $1 \leq i \leq \binom{n}{2}-1$, the inequalities 
\begin{equation} \label{Kro}
\sum_{\lambda \vdash n} \sum_{\mu \vdash n} b_{ \lambda, i-1} b_{\mu, i+1} g_{\lambda \mu}^{\nu} \leq \sum_{\lambda \vdash n} \sum_{\mu \vdash n} b_{\lambda, i} b_{\mu, i} g_{\lambda \mu}^{\nu}
\end{equation}
hold for all $\nu \vdash n$, where $g_{\lambda \mu}^{\nu}$ are the Kronecker coefficients.
\end{conj}

Conjecture \ref{kro} is intriguing because we know so little about Kronecker coefficients\footnote{One could also consider the somewhat unnatural log-concavity of the induced tensor product by replacing the Kronecker coefficients in (\ref{Kro}) with the Littlewood--Richardson coefficients, but the resulting inequalities are NOT true in general.}, and a major unsolved problem in representation theory and combinatorics is to find some combinatorial descriptions of Kronecker coefficients. Since $H^{*}(\mathcal{F}_{n}, \mathbb{C})$ carries the regular representation of $S_{n}$, each $g_{\lambda \mu}^{\nu}$ will actually appears in (\ref{Kro}).

Using the representation stability theory, we have the following
\begin{thm} \label{evi}
For all natural number $n$, the cohomology ring $H^{2 *}(\mathcal{F}_{n}, \mathbb{C})$ is equivariantly log-concave in degree $\leq3$ and co-degree $\leq3$ as graded representation of $S_{n}$. That is, 
\begin{equation*}
\begin{aligned}
&\left(H^{2*1}\right) ^{\otimes 2} \supseteq H^{2*0} \otimes H^{2*2}, \left(H^{2*(\binom{n}{2}-1)}\right) ^{\otimes 2} \supseteq H^{2*\binom{n}{2}} \otimes H^{2*(\binom{n}{2}-2)},\\ &\left(H^{2*2}\right) ^{\otimes 2} \supseteq H^{2*1} \otimes H^{2*3}, \left(H^{2*(\binom{n}{2}-2)}\right) ^{\otimes 2} \supseteq H^{2*(\binom{n}{2}-1)} \otimes H^{2*(\binom{n}{2}-3)},\\
&\left(H^{2*3}\right) ^{\otimes 2} \supseteq H^{2*2} \otimes H^{2*4}, \left(H^{2*(\binom{n}{2}-3)}\right) ^{\otimes 2} \supseteq H^{2*(\binom{n}{2}-2)} \otimes H^{2*(\binom{n}{2}-4)}
\end{aligned}
\end{equation*}
for all $n$. Equivalently, Conjecture \ref{Schur} and Conjecture \ref{kro} holds in these degrees for all $n$.
\end{thm}

\begin{proof}
In Sect.5 of \cite{church2015fi}, $H^{*}(\mathcal{F}_{n}, \mathbb{C}) \cong \mathbb{C}\left[t_{1}, \ldots, t_{n}\right] /\left(\sigma_{1}, \ldots, \sigma_{n}\right)$ is shown to be a co-FI-algebra of finite type, therefore each $\{H^{2*i}\}_{n}$ for fixed $i$ form a uniformly representation stable sequence of $S_{n}$-representations, see also \cite[Theorem 7.4]{church2013representation}. With a little more effort in analyzing the major index, it can be shown that for $i \leq 4$, $\{H^{2*i}\}_{n}$ stabilized once $n \geq 2i$. Therefore, by \cite[Lemma 3.2 and Theorem 3.3]{matherneequivariant}, to prove that
$H^{2*(m-1)} \otimes H^{2*(m+1)}$ is isomorphic to a subrepresentation of $H^{2*m} \otimes H^{2*m}$ for all $n$ and $m \leq 3$, it is sufficient to check that $H^{2*(m-1)} \otimes H^{2*(m+1)}$ is isomorphic to a sub-representation of $H^{2*m} \otimes H^{2*m}$ for $n \leq 4 m$.
We have performed these checks for $n \leq 12 $ using maple. The assertions for the co-degree part follow from the duality (\ref{dual}).
\end{proof}

\section{A unimodal conjecture}

By examining the computed data, we observe that the numerical evidence suggests the stronger conjecture
\begin{conj} \label{uni}
For all natural number $n$ and $1 \leq i \leq \binom{n}{2}-1$, the sequence $d_{\nu, i}$ is symmetric and unimodal, where $$
d_{\nu, i}  :=\sum_{\lambda \vdash n} \sum_{\mu \vdash n} b_{\lambda, i} b_{\mu, i} g_{\lambda \mu}^{\nu}-\sum_{\lambda \vdash n} \sum_{\mu \vdash n} b_{ \lambda, i-1} b_{\mu, i+1} g_{\lambda \mu}^{\nu}.
$$
\end{conj}

That is, when we decompose the virtual representation $H^{2*i}(\mathcal{F}_{n}) \otimes H^{2*i}(\mathcal{F}_{n})-H^{2*(i-1)}(\mathcal{F}_{n}) \otimes H^{2*(i+1)}(\mathcal{F}_{n})$, the multiplicity sequence $\{d_{\nu, i}\}$ should be symmetric and unimodal for all $n$ and all partition $\nu$ of $n$. The symmetry comes from the duality (\ref{dual}). Conjecture \ref{uni} implies Conjecture \ref{flag} (equivalently, Conjecture \ref{Schur} and Conjecture \ref{kro}) since $d_{\nu, 1}=d_{\nu,  \binom{n}{2}-1} \geq 0$ for all $\nu$ by Theorem \ref{evi}.

\begin{rmk}
In \cite{billey2020distribution} and  \cite{billey2020tableau}, the authors study the distribution of the major index on standard Young tableaux and enumerative questions involving the fake degrees, we note that the sequences of fake degrees ${b_{\lambda, i}}$ themselves are NOT always
unimodal.
\end{rmk}

\section{Some remarks}

\subsection{Other Coxeter groups}

Since the \emph{graded dimensions of the coinvariant ring are log-concave for any finite Coxeter group}, see \cite[Lemma 7.1.1 and Theorem 7.1.5]{bjorner2006combinatorics}, it is interesting to ask whether the coinvariant ring of Coxeter groups other than symmetric groups carries log-concave graded representations. Unfortunately, it is easy to show that except $m=2 \text{ or } 3$, all other dihedral groups $I_{2}(m)$ including the Weyl groups of type $B_{2}$ and type $G_{2}$ are not the case. We don't know whether the equivariant log-concavity is true for the simply laced (ADE) types. In \cite{wilson2014fiw}, the author proved that each sequence of graded piece of the coinvariant ring for the classical Weyl groups is uniformly representation stable, but we don't know whether the equivariant log-concavity holds after stabilizing

\subsection{Relation with the conjectures on the configuration spaces}

In \cite{matherneequivariant}, the authors made equivariant log
concavity conjectures for the cohomology or intersection cohomology of variant configuration spaces. In a series of papers \cites{atiyah2001configurations,atiyah2001equivariant,atiyah2002geometry,atiyah2002geometry2}, to answer the Berry-Robbins problem asked by two physicists, Sir Michael Atiyah constructed a continuous map from the configuration space $\text{ Conf} \left(n, \mathbb{R}^{3}\right)$ to the flag manifold $U_{n} / T$, which is compatible with the action of the symmetric group $S_{n}$ \footnote{See \cite{atiyah2002nahm} for the generalization to other Weyl groups.}. Both $H^{*}(\text{ Conf} \left(n, \mathbb{R}^{3}\right))$ and $H^{*}(U_{n} / T)$ carry the regular representation of $S_{n}$ \cite{atiyah2001equivariant}. But we can's get a formal relationship between log-concave conjectures on both sides using Atiyah's map. It seems that the unimodal property similar with Conjecture \ref{uni} are also true in their case \footnote{We thank Matthew H.Y. Xie for some verifications.}.

\subsection{The Springer representations}

One could ask a further question that whether the equivariant log-concavity hold for 
other total Springer representations in type A besides the full flag case. If this is true, it will imply that the Betti numbers of Springer fibers are unimodal and log-concave, which was observed in \cite{fresse2013unimodality} for some special types of partitions and $n \leq 9$. By Lusztig's result, the graded multiplicity is given by the Kostka--Foulkes polynomial, and the Frobenius image is therefore given by the modified Hall--Littlewood polynomials. Unfortunately, the equivariant log-convity of Springer representations are not true in general, the smallest counterexample appears in $S_{7}$, all counterexample up to $S_{10}$ is $\{(4,1,1,1), (5,1,1,1), (6,1,1,1), (5,2,1,1), (5,1,1,1,1), (7,1,1,1), (6,2,1,1),\\
(6,1,1,1,1), (5,2,1,1,1), (5,1,1,1,1,1)\}$. 

\bibliography{template}

\end{document}